\documentclass[british]{interact}
\usepackage[T1]{fontenc}
\usepackage[latin9]{inputenc}
\usepackage{amsmath}
\usepackage{amssymb}
\usepackage{esint}
\usepackage{url}
\usepackage{babel}
\usepackage[shortlabels]{enumitem}

\makeatletter

\newtheorem{thm}{Theorem}
\newtheorem{prop}{Proposition}
\newtheorem{coro}{Corollary}

\newtheorem{rem}{Remark} 
 
\newtheorem{defi}{Definition}

 \global\long\def\sbr#1{\left[ #1\right] }
 \global\long\def\cbr#1{\left\{  #1\right\}  }
 \global\long\def\rbr#1{\left(#1\right)}

 \global\long\def\P{\mathbb{P}}
 \global\long\def\R{\mathbb{R}}
 
 \global\long\def\F{\mathbb{F}}

 \global\long\def\dd#1{\textnormal{d}#1}

 \global\long\def\ab{[a,b]}
 \global\long\def\ra{\rightarrow}
 \global\long\def\TTV#1#2#3{{TV}^{#3}\!\rbr{#1,#2}}

 \global\long\def\UTV#1#2#3{{UTV}^{#3}\!\rbr{#1,#2}}
 \global\long\def\DTV#1#2#3{{DTV}^{#3}\!\rbr{#1,#2}}
 \global\long\def\ns{\infty}
 \global\long\def\f{:\left[a,b\right]\ra\R}

\begin{document}

\title{Quadratic variation of a c\`{a}dl\`{a}g semimartingale as a.s. limit of the normalized truncated variations}

\author{
\name{Rafa{\l{}} Marcin \L{}ochowski\thanks{CONTACT R.~M. \L{}ochowski. Email: rlocho@sgh.waw.pl}}
\affil{Department of Mathematics and Mathematical Economics, Warsaw School of Economics, ul. Madali\'{n}skiego 6/8, 02-513 Warsaw, Poland}
}

\maketitle

\begin{abstract}
For a real c\`{a}dl\`{a}g path $x$ we define sequence of semi-explicit quantities,
which \emph{do not depend on any partitions} and such that whenever
$x$ is a path of a c\`{a}dl\`{a}g semimartingale then these quantities tend
a.s. to the continuous part of the quadratic variation of the semimartingale. Next, we derive several consequences of this result and propose a new approach to define F\"{o}llmer's pathwise integral.
\end{abstract}
\smallskip


\begin{amscode}
\end{amscode}

\section{Introduction}
\label{intro}
In recent years there is significant interest in the pathwise approach
to stochastic calculus. One of the most important quantities in stochastic
calculus is arguably the quadratic variation of a semimartingale.
It is usually defined as the limit of sums of squares of the increments
of a semimartingale along sequence of deterministic partitions, as
the meshes of the partitions tend to $0,$ and the convergence holds
in probability. Unfortunately, when we allow random partitions, it
may happen that this convergence and the limit (if it exists) depend
on the partitions chosen (see for example \cite[Theorem 7.1]{Obloj_local:2015}).
As a result one may obtain different values of F\"{o}llmer's pathwise
integral \cite{Foellmer:1981} with respect to a specific path, along
different sequences of partitions.

Fortunately, when the partitions are obtained from stopping times
such that the osscillations of a path on the consecutive (half-open
on the right) intervals of these partitions tend a.s. (almost surely)
to $0,$ then there is no ambiguity and the sums of squares of the
increments along these partitions tend a.s. (or a.s. along some subsequence
of these partitions) to the quadratic variation (see for example \cite[Proposition 2.4 and Proposition 2.3 ]{Obloj_local:2015}).
One of such partition schemes dates back at least to Bichteler, see
\cite[Theorem 7.14]{Bichteler:1981}, \cite{Karandikar:1995} and
as a result one obtains a sequence of \emph{pathwise} sums of squares
of the increments which tend a.s. to the quadratic variation. Some
modification of this scheme, so called Lebesgue partitions, was proposed
by Vovk in \cite{Vovk_cadlag:2015}, to prove that the quadratic variation
of typical, model-free c\`{a}dl\`{a}g price paths with mildly restricted jumps
exists (along the Lebesgue partitions). Later, the same scheme was
used in \cite{LochPerkPro:2018} to prove the existence of the quadratic variation
of typical, model-free, c\`{a}dl\`{a}g price paths, with mildly restricted jumps
directed downward. Typical price paths are (roughly speaking) those trajectories representing possible evolution of prices of some asset
which do not allow to obtain infinite wealth by risking small amount and trading this asset (for formal definition see \cite{LochPerkPro:2018}).

Let $\mathbb{D}$ denote the family of c\`{a}dl\`{a}g functions $x:\left[0,+\ns\right)\ra\R.$
In this article, for any $x\in\mathbb{D}$ we will define another sequence
of semi-explicit quantities, which \emph{do not depend on any partitions}
and such that whenever $X_{t},$ $t\ge0,$ is a real c\`{a}dl\`{a}g semimartingale
on a filtered probability space $\left(\Omega,\F,\P\right)$ and that
usual conditions hold (see \cite[Chapt. I, Sect. 1]{Protter:2004uq}),
then for $x=X\left(\omega\right),$ $\omega\in\Omega,$ these quantities
tend $\P$-a.s. to the continuous part of the quadratic variation
of $X.$ This result is a generalisation of \cite[Theorem 1]{LochowskiMilosSPA:2013}
to the case of c\`{a}dl\`{a}g semimartingales, however, the proof will be
completely different from the proof of \cite[Theorem 1]{LochowskiMilosSPA:2013}. 
The approach used in this article will 
be similar to the old approach of Wong and Zakai \cite{WongZakai:1965}, where the authors 
replace semimartingale integrator by finite total variation (and piecewise linear) approximations of the integrator. However, 
contrary to the Wong-Zakai approach, we will use adapted 
approximations.
We will mainly use results of \cite{Lochowski_stoch_integral:2014}
and properties of so called double Skorohod map.
Next,  we will derive several consequences of this result and propose a new approach to define F\"{o}llmer's pathwise integral

In the sequel, refering to a c\`{a}dl\`{a}g semimartingale $X_{t},$ $t\ge0,$
we will always assume that $X$ is a semimartingale on a filtered
probability space such that usual conditions hold.

\section{Main result}
\label{main}

To state our main result we need several definitions.

Let $-\ns<a<b<+\ns$ and $x\f$ be a real-valued path. The \emph{truncated
variation} of $x$ with the truncation parameter $\varepsilon\geq0$ is defined
as 
\[
\TTV x{\ab}\varepsilon:=\sup_{n}\sup_{a\leq t_{0}<\ldots<t_{n}\leq b}\sum_{i=1}^{n}\max\left\{ \left|x\left(t_{i}\right)-x\left(t_{i-1}\right)\right|-\varepsilon,0\right\} .
\]
Thus $\TTV x{\ab}\varepsilon$ is obtained by taking supremum of sums of truncated
increments $\max\left\{ \left|x\left(t_{i}\right)-x\left(t_{i-1}\right)\right|-\varepsilon,0\right\} $
over \emph{all} possible partitions $\pi=\left\{ a\leq t_{0}<\ldots<t_{n}\leq b\right\} $
of $\left[a,b\right].$ It is possible to prove that $\TTV x{\ab}{\varepsilon}<+\ns$
for any $\varepsilon>0$ iff $x$ is regulated, i.e. it has finite left limits
$x\left(t-\right)$ for $t\in\left(a,b\right]$ and finite right limits
$x\left(t+\right)$ for $t\in\left[a,b\right)$, see \cite[Fact 2.2]{LochowskiColloquium:2013}.

Together with the truncated variation we define two companion quantities
- \emph{upward and downward truncated variations,} which are defined
in the following way: 
\[
\UTV x{\ab}{\varepsilon}:=\sup_{n}\sup_{a\leq t_{0}<\ldots<t_{n}\leq b}\sum_{i=1}^{n}\max\left\{ x\left(t_{i}\right)-x\left(t_{i-1}\right)-\varepsilon,0\right\}
\]
and 
\[
\DTV x{\ab}{\varepsilon}:=\sup_{n}\sup_{a\leq t_{0}<\ldots<t_{n}\leq b}\sum_{i=1}^{n}\max\left\{ x\left(t_{i-1}\right)-x\left(t_{i}\right)-\varepsilon,0\right\} .
\]
\begin{rem}

The definitions of the (upward-, downward-) truncated variation may
seem ``pulled out of a hat'', however, these three quantities have
very natural interpretation: they are (attainable) lower bounds for
the total (resp. positive, negative) variation of any path approximating
$x$ with the accuracy $\varepsilon/2,$ see \cite[displays (2.1)-(2.3)]{LochowskiGhomrasni:2014}.

\end{rem}

Now we are ready to state our main result.

\begin{thm}\label{main} Let $X_{t},$ $t\geq0,$ be a real c\`{a}dl\`{a}g
semimartingale on a filtered probability space $\left(\Omega,\F,\P\right)$
such that usual conditions hold. For each $\varepsilon>0$ and  $ t\geq0$ let us define the
following c\`{a}dl\`{a}g processess 
\[
T_{t}^{\varepsilon}:=\varepsilon\cdot\TTV X{\left[0,t\right]}{\varepsilon},\quad U_{t}^{\varepsilon}:=\varepsilon\cdot\UTV X{\left[0,t\right]}{\varepsilon}\hbox{ and }D_{t}^{\varepsilon}:=\varepsilon\cdot\DTV X{\left[0,t\right]}{\varepsilon},
\]
then 
\[
\rbr{T^{\varepsilon},U^{\varepsilon},D^{\varepsilon}}\Rightarrow\rbr{\left[X\right]^{cont},\frac{1}{2}\left[X\right]^{cont},\frac{1}{2}\left[X\right]^{cont}}\quad\hbox{as }{\varepsilon}\ra0+,
\]
where ``$\Rightarrow$'' denotes $\P$-a.s. convergence in the uniform
convergence topology on compact subsets of positive half-line $\left[0,+\ns\right)$
and $\left[X\right]^{cont}$ denotes the continuous part of the quadratic
variation of $X.$

\end{thm}

\begin{proof} 

{\bf (I) Proof of the convergence of $T^{\varepsilon}.$}

For $t\geq0$ let $\Delta X_{t}=X_{t}-X_{t-},$ where
$X_{0-}:=0$ and for $t>0,$ $X_{t-}=\lim_{s\ra t-}X_{s},$ be the
jump at the moment $t.$ For each $\varepsilon>0$ let $X^{\varepsilon}$ be the process
constructed (for given process $X$) as in \cite[Section 2, display (2.1)]{Lochowski_stoch_integral:2014}. This construction is related to the double Skorohod map $\Gamma^{\varepsilon}$ on $[-{\varepsilon},{\varepsilon}],$ see \cite{Lochowski_stoch_integral:2014}, \cite{Burdzy:2009lr}, via $X^{\varepsilon} = X - \Gamma^{\varepsilon}\rbr{X-X_0}.$ $X^{\varepsilon}$ has the following properties. 
\begin{enumerate}[1.]
\item $X^{\varepsilon}$ has locally finite total variation; 
\item $X^{\varepsilon}$ has c\`{a}dl\`{a}g paths; 
\item for every $t\geq0,$ $\left|X_{t}-X_{t}^{\varepsilon}\right|\leq {\varepsilon};$ 
\item for every $t>0,$ $\left|\Delta X_{t}^{\varepsilon}\right|\leq\left|\Delta X_{t}\right|;$ 
\item the process $X^{\varepsilon}$ is adapted to the filtration $\F;$ 
\item $X_{0}^{\varepsilon}=X_{0}.$ 
\end{enumerate}
Moreover, by \cite[Lemma 5.1]{Lochowski_stoch_integral:2014} for
any $t\geq0$ we have 
\begin{equation}
\TTV X{\left[0,t\right]}{2{\varepsilon}}\leq\TTV{X^{\varepsilon}}{\left[0,t\right]}{}\leq\TTV X{\left[0,t\right]}{2{\varepsilon}}+2{\varepsilon},\label{eq:cTVlim1}
\end{equation}
where $\TTV{X^{\varepsilon}}{\left[0,t\right]}{}:=\TTV{X^{\varepsilon}}{\left[0,t\right]}0$
denotes the total variation of $X^{\varepsilon}$ on $\left[0,t\right].$

Recall the classical Jordan decomposition and notice that $\UTV{X^{\varepsilon}}{\left[0,t\right]}{}:=\UTV{X^{\varepsilon}}{\left[0,t\right]}0$
and $\DTV{X^{\varepsilon}}{\left[0,t\right]}{}:=\DTV{X^{\varepsilon}}{\left[0,t\right]}0$
are nothing else but positive and negative parts of the total variation
of $X^{\varepsilon}.$ By \cite[Lemma 5.2]{Lochowski_stoch_integral:2014} we
have that $\dd{\UTV{X^{\varepsilon}}{\left[0,t\right]}{}}$ and $\dd{\DTV{X^{\varepsilon}}{\left[0,t\right]}{}}$
are mutually singular measures carried by $\left\{ t>0:X_{t}-X_{t}^{\varepsilon}={\varepsilon}\right\} $
and $\left\{ t>0:X_{t}-X_{t}^{\varepsilon}=-{\varepsilon}\right\} ,$ and on these sets
we have 
\[
\dd{\UTV{X^{\varepsilon}}{\left[0,t\right]}{}}=\dd{X^{\varepsilon}}\quad\mbox{and}\quad\dd{\DTV{X^{\varepsilon}}{\left[0,t\right]}{}}=-\dd{X^{\varepsilon}}.
\]
From all this it follows that (see also \cite[display (5.2)]{Lochowski_stoch_integral:2014})
\begin{align}
{\varepsilon}\cdot\TTV{X^{\varepsilon}}{\left[0,t\right]}{} & ={\varepsilon}\int_{0}^{t}\dd{\TTV{X^{\varepsilon}}{\left[0,s\right]}{}}=\int_{0}^{t}\left(X-X^{\varepsilon}\right)\dd{X^{\varepsilon}}.\label{eq:TV_representation}
\end{align}

Representation
(\ref{eq:TV_representation}) together with the estimates (\ref{eq:cTVlim1})
will be the main ingredients of the proof. 

Setting $X_{0-}^{\varepsilon}:=0$
we calculate 
\begin{align}
{\varepsilon}\cdot\TTV{X^{\varepsilon}}{\left[0,t\right]}{} & =\int_{0}^{t}\left(X-X^{\varepsilon}\right)\dd{X^{\varepsilon}}=\int_{0}^{t}\left(X_{-}-X_{-}^{\varepsilon}+\Delta\left(X-X^{\varepsilon}\right)\right)\dd{X^{\varepsilon}}\nonumber \\
 & =\int_{0}^{t}X_{-}\dd{X^{\varepsilon}}-\int_{0}^{t}X_{-}^{\varepsilon}\dd{X^{\varepsilon}}+\sum_{0<s\leq t}\Delta\left(X_{s}-X_{s}^{\varepsilon}\right)\Delta X^{\varepsilon}.\label{eq:cTVlim2}
\end{align}

Let us now fix $T>0$ and for $n=1,2,\ldots,$ let us define ${\varepsilon}(n)=1/(2n).$
By \cite[Theorem 3.2]{Lochowski_stoch_integral:2014} we have 
\begin{equation}
\lim_{n\ra+\ns}\sup_{0\le t\le T}\left|\int_{0}^{t}X_{-}\dd{X^{{\varepsilon}(n)}}-\int_{0}^{t}X_{-}\dd X-\left[X\right]_{t}^{cont}\right|=0\quad\mbox{a.s.},\label{eq:Theorem3.2}
\end{equation}
where $\int_{0}^{t}X_{-}\dd{X^{{\varepsilon}(n)}}$ denotes the Lebesgue-Stieltjes
integral (recall that $X^{{\varepsilon}(n)}$ has finite total variation)
and $\int_{0}^{t}X_{-}\dd X$ denotes the (semimartingale) stochastic
integral. We may rewrite (\ref{eq:Theorem3.2}) in the following way:
\begin{align*}
\int_{0}^{t}X_{-}\dd{X^{{\varepsilon}(n)}} & \Rightarrow\int_{0}^{t}X_{-}\dd X+\left[X\right]_{t}^{cont}=\frac{1}{2}\left(X_{t}^{2}-X_{0}^{2}-\left[X\right]_{t}\right)+\left[X\right]_{t}^{cont}
\end{align*}
for $t\in\left[0,T\right],$ where $\left[X\right]_{t}=\left[X\right]_{t}^{cont}+\sum_{0<s\leq t}\left(\Delta X_{s}\right)^{2}$
denotes the quadratic variation of $X.$ Next, by the integration
by parts formula for the Lebesgue-Stieltjes integral for any ${\varepsilon}>0$
and $t\in\left[0,T\right]$ we calculate 
\[
\int_{0}^{t}X_{-}^{\varepsilon}\dd{X^{\varepsilon}}=\frac{1}{2}\left(\left(X_{t}^{\varepsilon}\right)^{2}-\left(X_{0}^{\varepsilon}\right)^{2}-\sum_{0<s\leq t}\left(\Delta X_{s}^{\varepsilon}\right)^{2}\right).
\]
Also, by properties 3. and 4. satisfied by $X^{\varepsilon}$ and by the dominated
convergence we have 
\[
\sup_{0\leq t\leq T}\sum_{0<s\leq t}\left|\left(\Delta X_{s}\right)^{2}-\left(\Delta X_{s}^{\varepsilon}\right)^{2}\right| = \sum_{0<s\leq T}\left|\left(\Delta X_{s}\right)^{2}-\left(\Delta X_{s}^{\varepsilon}\right)^{2}\right|\ra0
\]
and 
\begin{align*}
\sup_{0\leq t\leq T}\sum_{0<s\leq t}\left|\Delta\left(X_{s}-X_{s}^{\varepsilon}\right)\Delta X_{s}^{\varepsilon}\right| & \leq\sum_{0<s\leq T}\min\left\{ 2{\varepsilon}\left|\Delta X_{s}\right|,2\left|\Delta X_{s}\right|^{2}\right\} \ra0
\end{align*}
as ${\varepsilon}\ra0+,$ where ``$\ra$'' denotes $\P$-a.s. convergence. From
(\ref{eq:cTVlim2}) and last four relations we get 
\begin{align}
 & {\varepsilon}(n)\cdot\TTV{X^{{\varepsilon}(n)}}{\left[0,t\right]}{} \nonumber  \\
 & =\int_{0}^{t}X_{-}\dd{X^{{\varepsilon}(n)}}-\int_{0}^{t}X_{-}^{{\varepsilon}(n)}\dd{X^{{\varepsilon}(n)}}+\sum_{0<s\leq t}\Delta\left(X_{s}-X_{s}^{{\varepsilon}(n)}\right)\Delta X^{{\varepsilon}(n)} \nonumber \\
 & \Rightarrow\frac{1}{2}\left[X\right]_{t}^{cont}. \label{eq:cTVlim3}
\end{align}
Hence, from (\ref{eq:cTVlim1}) and (\ref{eq:cTVlim3})  we get for some $r(n) \in [-\rbr{2\varepsilon(n)}^2,0] $
\begin{eqnarray*}
 \frac{1}{n}\cdot\TTV X{\left[0,t\right]}{1/n} & = & 2{\varepsilon}(n)\cdot\TTV X{\left[0,t\right]}{2{\varepsilon}(n)} \\
& = & 2{\varepsilon}(n)\cdot\TTV{X^{{\varepsilon}(n)}}{\left[0,t\right]}{} + r(n) \Rightarrow\left[X\right]_{t}^{cont}.
\end{eqnarray*}

Finally,  the convergence ${\varepsilon}_{n}\cdot\TTV X{\left[0,t\right]}{{\varepsilon}_{n}}\Rightarrow\left[X\right]_{t}^{cont}$
for any sequence ${\varepsilon}_{n}\ra0+$ follows from the estimates 
\begin{align*}
& \frac{\left\lfloor 1/{\varepsilon}_{n}\right\rfloor }{\left\lfloor 1/{\varepsilon}_{n}\right\rfloor +1}\frac{1}{\left\lfloor 1/{\varepsilon}_{n}\right\rfloor }\cdot\TTV X{\left[0,t\right]}{1/\left\lfloor 1/{\varepsilon}_{n}\right\rfloor } \\ & \leq {\varepsilon}_{n}\cdot\TTV X{\left[0,t\right]}{{\varepsilon}_{n}}
 \leq\frac{\left\lceil 1/{\varepsilon}_{n}\right\rceil }{\left\lceil 1/{\varepsilon}_{n}\right\rceil -1}\frac{1}{\left\lceil 1/{\varepsilon}_{n}\right\rceil }\cdot\TTV X{\left[0,t\right]}{1/\left\lceil 1/{\varepsilon}_{n}\right\rceil }
\end{align*}
valid for ${\varepsilon}_{n}<1,$ which stem directly from inequalities 
\[
\frac{1}{\left\lfloor 1/{\varepsilon}_{n}\right\rfloor +1}\le {\varepsilon}_{n}\le\frac{1}{\left\lceil 1/{\varepsilon}_{n}\right\rceil -1}\mbox{ and }\frac{1}{\left\lceil 1/{\varepsilon}_{n}\right\rceil }\le {\varepsilon}_{n}\le\frac{1}{\left\lfloor 1/{\varepsilon}_{n}\right\rfloor }
\]
(valid for ${\varepsilon}_{n}<1$), and the fact that the function $\left(0,+\ns\right)\ni {\varepsilon}\mapsto\TTV X{\left[0,t\right]}{\varepsilon}$
is non-increasing.

{\bf (II) Proof of the convergence of the whole triplet $\rbr{T^{\varepsilon},U^{\varepsilon},D^{\varepsilon}}.$}

To prove the convergence of the whole triplet $\rbr{T^{\varepsilon},U^{\varepsilon},D^{\varepsilon}}$
for any ${\varepsilon}>0$ let us define the auxilary process 
\begin{equation}
\tilde{X}_{t}^{\varepsilon}:=X_{0}+\UTV X{\left[0,t\right]}{\varepsilon}-\DTV X{\left[0,t\right]}{\varepsilon}.\label{eq:Jordan1}
\end{equation}
The process $\tilde{X}^{\varepsilon}$ uniformly approximates $X$ with accuracy
$c.$ This is the consequence of \cite[Theorem 4]{LochowskiGhomrasniMMAS:2015}
and the classical Jordan decomposition. Indeed, let us fix some $\omega\in\Omega.$
By \cite[Theorem 4]{LochowskiGhomrasniMMAS:2015} for c\`{a}dl\`{a}g $x=X\left(\omega\right)$
there exists some picewise monotone $x^{\varepsilon}:\left[0,+\ns\right)\ra\R$
such that $x^{\varepsilon}$ approximates $x$ with accuracy ${\varepsilon}/2$ and 
\begin{equation}
\UTV x{\left[0,t\right]}{\varepsilon}=\UTV{x^{\varepsilon}}{\left[0,t\right]}{},\quad\DTV x{\left[0,t\right]}{\varepsilon}=\DTV{x^{\varepsilon}}{\left[0,t\right]}{}.\label{eq:Jordan}
\end{equation}
Thus, by the classical Jordan decomposition, 
\begin{align*}
x^{\varepsilon}\left(t\right) & =x^{\varepsilon}\left(0\right)+\UTV{x^{\varepsilon}}{\left[0,t\right]}{}-\DTV{x^{\varepsilon}}{\left[0,t\right]}{}\\
 & =x^{\varepsilon}\left(0\right)+\UTV x{\left[0,t\right]}{\varepsilon}-\DTV x{\left[0,t\right]}{\varepsilon}.
\end{align*}
Since $x^{\varepsilon}$ approximates $x=X\left(\omega\right)$ with accuracy
$c/2$ we must have that $\left|X_{0}\left(\omega\right)-x^{\varepsilon}\left(0\right)\right|\le {\varepsilon}/2.$
From this, the definition of $\tilde{X}_{t}^{\varepsilon}$ and the triangle
inequality we get 
\begin{align}
\left|\tilde{X}_{t}^{\varepsilon}\left(\omega\right)-X_{t}\left(\omega\right)\right| & \le\left|\tilde{X}_{t}^{\varepsilon}\left(\omega\right)-x^{\varepsilon}\left(t\right)\right|+\left|x^{\varepsilon}\left(t\right)-X_{t}\left(\omega\right)\right|\nonumber \\
 & =\left|X_{0}\left(\omega\right)-x^{\varepsilon}\left(0\right)\right|+\left|x^{\varepsilon}\left(t\right)-X_{t}\left(\omega\right)\right|\nonumber \\
 & \le {\varepsilon}/2+{\varepsilon}/2={\varepsilon}.\label{eq:unif_approx}
\end{align}
From \cite[Theorem 4]{LochowskiGhomrasniMMAS:2015} and (\ref{eq:Jordan})
we also have the relation 
\begin{align}
\TTV{X\left(\omega\right)}{\left[0,t\right]}{\varepsilon} & =\TTV{x^{\varepsilon}}{\left[0,t\right]}{}\nonumber \\
 & =\UTV{x^{\varepsilon}}{\left[0,t\right]}{}+\DTV{x^{\varepsilon}}{\left[0,t\right]}{}\nonumber \\
 & =\UTV{X\left(\omega\right)}{\left[0,t\right]}{\varepsilon}+\DTV{X\left(\omega\right)}{\left[0,t\right]}{\varepsilon}.\label{eq:Jordan2}
\end{align}
Finally, from (\ref{eq:unif_approx}) we get that $\tilde{X}_{t}^{\varepsilon}=X_{t}+R_{t}^{\varepsilon},$
where $\left|R_{t}^{\varepsilon}\right|\leq {\varepsilon}$ for $t\geq0$ and from (\ref{eq:Jordan1})
and (\ref{eq:Jordan2}) we have the following representation 
\[
\UTV X{\left[0,t\right]}{\varepsilon}=\frac{1}{2}\rbr{\TTV X{\left[0,t\right]}{\varepsilon}+X_{t}-X_{0}+R_{t}^{\varepsilon}},
\]
\[
\DTV X{\left[0,t\right]}{\varepsilon}=\frac{1}{2}\rbr{\TTV X{\left[0,t\right]}{\varepsilon}-X_{t}+X_{0}-R_{t}^{\varepsilon}}.
\]
From this representation we obtain the convergence of the whole triplet
$\rbr{T^{\varepsilon},U^{\varepsilon},D^{\varepsilon}}.$
\end{proof}

\section{Some consequences of Theorem \ref{main}}

From Theorem \ref{main} we also obtain pathwise formulas which in the limit tend a.s. to the quadratic covariation of two semimartingales. We have
 
\begin{coro}
Let $X_t$ and $Y_t,$ $t \ge 0,$ be two real c\`{a}dl\`{a}g semimartingales, then 
\[
{\varepsilon} \cdot \cbr{\TTV {X+Y}{\left[0,t\right]}{\varepsilon} -\TTV {X-Y}{\left[0,t\right]}{\varepsilon}} \Rightarrow 4 \left[X,Y\right]^{cont}_t, \hbox{ as } {\varepsilon} \ra 0+,
\]
 where $\left[X,Y\right]^{cont}$ denotes the continuous part of the quadratic covariation of $X$ and $Y,$ i.e. 
 $\left[X,Y\right]_t = \left[X,Y\right]^{cont}_t + \sum_{0<s\le t} \Delta X_s \Delta Y_s$ for $t \ge 0.$
\end{coro}

Also, using recent result \cite[Theorem 1]{LochowskiColloquium:2017}
and Theorem \ref{main} we obtain another pathwise formula for the
quadratic variation, where the numbers of interval crossings play a
role. To state this result we need to introduce the numbers of times
the graph of regulated $x:\left[a,b\right]\ra\R$ ``(down-, up-)
crosses'' (on $\left[a,b\right]$) the closed value interval $[y,y+{\varepsilon}].$
\begin{defi}\label{defd} Given a function $x:\left[a,b\right]\rightarrow\mathbb{R},$
for ${\varepsilon}\geq0$ we put $\sigma_{0}^{\varepsilon}=a$ and for $n=0,1,...$ 
\[
\tau_{n}^{\varepsilon}=\inf\left\{ t\ge\sigma_{n}^{\varepsilon}:t\leq b,x(t)>y+{\varepsilon}\right\} ,\mbox{ }\sigma_{n+1}^{\varepsilon}=\inf\left\{ t\ge\tau_{n}^{\varepsilon}:t\leq b,x(t)<y\right\} .
\]
Next, we set 
\[
d_{\varepsilon}^{y}\left(x,\left[a,b\right]\right):=\max\left\{ n:\sigma_{n}^{\varepsilon}\leq b\right\} .
\]
\end{defi} Similarly we define. \begin{defi}\label{defu} Given
a function $x:\left[a,b\right]\rightarrow\mathbb{R},$ for $c\geq0$
we put ${\tilde{\sigma}}_{0}^{\varepsilon}=a$ and for $n=0,1,...$ 
\[
{\tilde{\tau}}_{n}^{\varepsilon}=\inf\left\{ t \ge {\tilde{\sigma}}_{n}^{\varepsilon}:t\leq b,x(t)<y\right\} ,\mbox{ }{\tilde{\sigma}}_{n+1}^{\varepsilon}=\inf\left\{ t\ge{\tilde{\tau}}_{n}^{\varepsilon}:t\leq b,x(t)> y+{\varepsilon}\right\} .
\]
Next, we set 
\begin{equation}
u_{\varepsilon}^{y}\left(x,\left[a,b\right]\right):=\max\left\{ n:{\tilde{\sigma}}_{n}^{\varepsilon}\leq b\right\} .\label{eq:u_def}
\end{equation}
\end{defi} In all definitions we apply the convention that $\inf\emptyset=+\ns.$

The number $d_{\varepsilon}^{y}\left(x,\left[a,b\right]\right)$ can be viewed
as the number of times the graph of $x$ ``downcrosses'' (on $\left[a,b\right]$)
the closed value interval $[y,y+{\varepsilon}],$ while the number $u_{\varepsilon}^{y}\left(x,\left[a,b\right]\right)$
can be viewed as the number of times the graph of $x$ ``upcrosses''
the value interval $[y,y+{\varepsilon}].$

At last, for $x$ and the interval $\left[a,b\right]$ as in two preceding
definitions, we define the number of times the graph of $x$ crosses
(on $\left[a,b\right]$) the value interval $[y,y+{\varepsilon}]$ as 
\[
n_{\varepsilon}^{y}\left(x,\left[a,b\right]\right):=d_{\varepsilon}^{y}\left(x,\left[a,b\right]\right)+u_{\varepsilon}^{y}\left(x,\left[a,b\right]\right).
\]
\cite[Theorem 1]{LochowskiColloquium:2017} states that 
\begin{equation}
\UTV x{\left[a,b\right]}{\varepsilon}=\int_{\R}u_{\varepsilon}^{y}\left(x,\left[a,b\right]\right)\dd y,\label{eq:estim1}
\end{equation}
\begin{equation}
\DTV x{\left[a,b\right]}{\varepsilon}=\int_{\R}d_{\varepsilon}^{y}\left(x,\left[a,b\right]\right)\dd y\label{eq:estimdtv}
\end{equation}
and 
\begin{equation}
\TTV x{\left[a,b\right]}{\varepsilon}=\int_{\R}n_{\varepsilon}^{y}\left(x,\left[a,b\right]\right)\dd y.\label{eq:estimtv}
\end{equation}
Using (\ref{eq:estim1})-(\ref{eq:estimtv}) and Theorem \ref{main}
we get 
\begin{coro} Let $X_t$, $t \ge 0,$ be a real c\`{a}dl\`{a}g semimartingale, then 
\[
\int_{\R}{\varepsilon}\cdot u_{\varepsilon}^{y}\left(X,\left[0,\cdot\right]\right)\dd y\Rightarrow\frac{1}{2}\left[X\right]^{cont},
\]
\[
\int_{\R}{\varepsilon}\cdot d_{\varepsilon}^{y}\left(X,\left[0,\cdot\right]\right)\dd y\Rightarrow\frac{1}{2}\left[X\right]^{cont},
\]
\[
\int_{\R}{\varepsilon}\cdot n_{\varepsilon}^{y}\left(X,\left[0,\cdot\right]\right)\dd y\Rightarrow\left[X\right]^{cont}
\]
as ${\varepsilon}\ra0+.$ 
\end{coro}
Much stronger result of this type, namely that ${\varepsilon}\cdot n_{\varepsilon}^{y}\left(X,\left[0,\cdot\right]\right)$
tends a.s. to the local time of $X$ at all but countably many real
$y$s, was proven in \cite[Theorem 3.3]{Lemieux:1983}, but only for
semimartingales satisfying the condition $\sum_{0<s\le t}\left|\Delta X_{s}\right|<+\ns$
a.s.

Theorem \ref{main} and the construction used in its proof imply also
\begin{coro}
Any real c\`{a}dl\`{a}g semimartingale $X$ may be uniformly approximated with accuracy
$\varepsilon$ by finite variation and adapted (to the natural filtration of
$X$) processes, whose total variation is of order $O\rbr{{\varepsilon}^{-1}}$ as ${\varepsilon} \ra 0+.$ Moreover,
if $X$ is a pure-jump semimartingale, then it may be uniformly approximated
with accuracy ${\varepsilon}$ by finite variation, adapted processes whose total
variation is of order $o\rbr{{\varepsilon}^{-1}}$ as ${\varepsilon} \ra 0+.$
\end{coro}
\begin{rem}
If $X$ is a strictly $\alpha$-stable process, $\alpha\in\left(1,2\right),$
using scaling properties of $X$ it may be proven that $\TTV X{\left[0,T\right]}{\varepsilon}$
is of order ${\varepsilon}^{1-\alpha}$ as ${\varepsilon}\ra 0+$. 

However, there exist a pure-jump semimartingale $X$
for which $\TTV X{\left[0,T\right]}{\varepsilon}$
is of order greater than ${\varepsilon}^{-\beta}$ as ${\varepsilon}\ra 0+$ for any $\beta<1.$
An example of such a semimartingale is given in \cite[Proposition 3(a)]{Lepingle:1976}. That for this $X,$ $\TTV X{\left[0,T\right]}{\varepsilon}$
is of order greater than ${\varepsilon}^{-\beta}$ (as ${\varepsilon}\ra 0+$) for any $\beta<1$ follows from the fact that $X$ has a.s. infinite $2$-variation
norm and from \cite[Proposition 2]{LochowskiJIA:2018}.
\end{rem}

\section{F\"{o}llmer's pathwise integral}

Inspired by equation (2) of Section 1, in this section we will define
an integral with respect to a c\`{a}dl\`{a}g path $x:[0,+\infty)\ra\R$ which
may be uniformly approximated for any ${\varepsilon}>0$ with accuracy ${\varepsilon}$ by
some c\`{a}dl\`{a}g path $x^{\varepsilon}:[0,+\infty)\ra\R$ with finite total variation on compacts and such that
the \emph{uniform} (on compacts) limit 
\begin{equation}
\left\langle x\right\rangle _{t}:=2\lim_{{\varepsilon}\ra0+}\int_{0}^{t}\left(x_s-x^{\varepsilon}_s\right)\mathrm{d}x^{\varepsilon}_s\label{eq:q_var}
\end{equation}
exists for any $t\ge0.$ The integral in (\ref{eq:q_var}) is understood
as the classical Lebesgue-Stieltjes integral. Note that since $\sup_{t\ge 0} \left| x_t - x^{\varepsilon}_t\right| \le {\varepsilon}$, the function $t \mapsto \left\langle x\right\rangle _t^{\varepsilon}:=2\int_{0}^{t}\left(x-x^{\varepsilon}\right)\mathrm{d}x^{\varepsilon}$, $t \ge 0$, has  on the interval $[0,t]$ only jumps whose absolute values are no greater than $2{\varepsilon} \cdot \rbr{\sup_{0<s\le t}\left|\Delta x_s \right|+2{\varepsilon}}$ and thus the limit function
$t \mapsto \left\langle x\right\rangle _{t}$
is continuous. 

Let ${\cal X}=\left(x^{\varepsilon}\right)_{{\varepsilon}>0}$ be the family of functions
$x^{\varepsilon}.$ Now, for a measurabe function $f:\R\ra\R$ and $t>0$ we define 
\[
({\cal X})\int_{0+}^{t}f\left(x_{s-}\right)\mathrm{d}x_{s}:=\lim_{{\varepsilon}\ra0+}\int_{0+}^{t}f\left(x_{s-}\right)\mathrm{d}x_{s}^{\varepsilon}
\]
if this limit exists. We have the following result similar to \cite[TH\'EOR\`EME]{Foellmer:1981}
\begin{thm}\label{Foelmer}
Let $x:[0,+\infty)\ra\R$ be a c\`{a}dl\`{a}g path such that $\sum_{0<s\le t}\left(\Delta x_{s}\right)^{2}<+\infty$
for any $t > 0.$ Assume that the family ${\cal X}=\left(x^{\varepsilon}\right)_{{\varepsilon}>0}$ of c\`{a}dl\`{a}g paths $x^{\varepsilon}:[0,+\infty)\ra\R$
is such that $\left\Vert x-x^{\varepsilon}\right\Vert _{\infty}:=\sup_{s\ge0}\left|x_s-x^{\varepsilon}_s\right|\le {\varepsilon}$, $x^{\varepsilon}$ has finite total variation on compacts
and the uniform (on compacts) limit (\ref{eq:q_var}) exists. Moreover, assume that 
\begin{equation}
\sup_{{\varepsilon}>0}\int_{0}^{t}\left|x_s-x^{\varepsilon}_s\right|\left|\mathrm{d}x^{\varepsilon}_s\right|<+\infty\label{eq:boundedness}
\end{equation}
and there exists a constant $K$ such that 
\begin{equation}
\left|\Delta x_{t}^{\varepsilon}\right|\le K\left|\Delta x_{t}\right|\label{eq:est_jumps}
\end{equation}
for any $t >0.$ Then, for any $f:\R\ra\R$ of class $C^{1}$ and $t>0$ the integral $({\cal X})\int_{0+}^{t}f\left(x_{s-}\right)\mathrm{d}x_s$
exists, moreover, we have the following formula 
\begin{eqnarray*}
F\left(x_{t}\right)-F\left(x_{0}\right) & = & ({\cal X})\int_{0+}^{t}f\left(x_{s-}\right)\mathrm{d}x_{s}-\frac{1}{2}\int_{0+}^{t}f'\left(x_{s-}\right)\mathrm{d}\left\langle x\right\rangle _{s}\\
 &  & +\sum_{0<s\le t}\left\{ F\left(x_{s}\right)-F\left(x_{s-}\right)-f\left(x_{s-}\right)\Delta x_{s}\right\} ,
\end{eqnarray*}
where $F$ is an antiderivative of $f$ and $\int_{0+}^{t}f'\left(x_{s-}\right)\mathrm{d}\left\langle x\right\rangle _{s}$ is the usual Lebesgue-Stieltjes integral.
\end{thm}
\begin{proof} For any ${\varepsilon}>0$ and $s \ge 0$ we write $f\left(x_{s}\right)-f\left(x_{s}^{\varepsilon}\right)=f'\left(\tilde{x}_{s}^{\varepsilon}\right)\left(x_{s}-x_{s}^{\varepsilon}\right)$
for some $\tilde{x}_{s}^{\varepsilon}\in\left[\min\left\{ x_{s},x_{s}^{\varepsilon}\right\} ,\max\left\{ x_{s},x_{s}^{\varepsilon}\right\} \right].$
Thus we have 
\begin{eqnarray}
\int_{0+}^{t}f\left(x_{s-}\right)\mathrm{d}x_{s}^{\varepsilon} & = & \int_{0+}^{t}f\left(x_{s}\right)-\Delta f\left(x_{s}\right)\mathrm{d}x_{s}^{\varepsilon} \label{eq:jedne} \\
 & = & \int_{0+}^{t}f'\left(\tilde{x}_{s}^{\varepsilon}\right)\left(x_{s}-x_{s}^{\varepsilon}\right)\mathrm{d}x_{s}^{\varepsilon}+\int_{0+}^{t}f\left(x_{s}^{\varepsilon}\right)\mathrm{d}x_{s}^{\varepsilon}-\sum_{0<s\le t}\Delta f\left(x_{s}\right)\Delta x_{s}^{\varepsilon}. \nonumber
\end{eqnarray}
Let us notice that $f'\left(\tilde{x}_{s}^{\varepsilon}\right)\left(x_{s}-x_{s}^{\varepsilon}\right)$ is measurable since it is equal $f\left(x_{s}\right)-f\left(x_{s}^{\varepsilon}\right)$. Since $\tilde{x}_{s}^{\varepsilon}\in\left[\min\left\{ x_{s},x_{s}^{\varepsilon}\right\} ,\max\left\{ x_{s},x_{s}^{\varepsilon}\right\} \right]$ we have the estimate 
\begin{align}
&\left| \int_{0+}^{t}f'\left(\tilde{x}_{s}^{\varepsilon}\right)\left(x_{s}-x_{s}^{\varepsilon}\right)\mathrm{d}x_{s}^{\varepsilon} - \int_{0+}^{t}f'\left({x}_{s}\right)\left(x_{s}-x_{s}^{\varepsilon}\right)\mathrm{d}x_{s}^{\varepsilon} \right| \nonumber \\
& \le \rbr{\sup_{y,z \in \sbr{A_t-{\varepsilon},B_t+{\varepsilon}}, |y-z| \le {\varepsilon}} \left| f'(y) - f'(z)\right| } \int_{0+}^{t}\left|x-x^{\varepsilon}\right|\left|\mathrm{d}x^{\varepsilon}\right|, \label{relation}
\end{align}
where $A_t = \inf_{s \in [0,t]} x_s$ and $B_t = \sup_{s \in [0,t]} x_s$. By the continuity of $f'$ and (\ref{eq:boundedness})
 the right side of (\ref{relation}) tends to $0$ as ${\varepsilon} \ra 0+$.
Again by the continuity of $f',$ we may replace $f'\left({x}_{s}\right)$ in $\int_{0}^{t}f'\left({x}_{s}\right)\left(x_{s}-x_{s}^{\varepsilon}\right)\mathrm{d}x_{s}^{\varepsilon}$ on the left side of (\ref{relation}) by a picewise constant function uniformly approximating $f'\left({x}_{s}\right)$ with arbitrary accuracy and using (\ref{relation}),  condition (\ref{eq:boundedness})
 and (\ref{eq:q_var}) we get 
\begin{equation}
\int_{0+}^{t}f'\left(\tilde{x}_{s}^{\varepsilon}\right)\left(x_{s}-x_{s}^{\varepsilon}\right)\mathrm{d}x_{s}^{\varepsilon}\rightarrow\frac{1}{2}\int_{0+}^{t}f'\left(x_{s}\right)\mathrm{d}\left\langle x\right\rangle _{s}=\frac{1}{2}\int_{0+}^{t}f'\left(x_{s-}\right)\mathrm{d}\left\langle x\right\rangle _{s}, \label{relation1}
\end{equation} 
as ${\varepsilon} \ra 0+$, where the last equality follows from the continuity of the function
$t \mapsto \left\langle x\right\rangle _{t}.$
Note also that due to (\ref{eq:boundedness}), $\left\langle x\right\rangle $
has finite total variation on compacts. Next, from the properties of the Lebesgue-Stieltjes
integral we obtain (recall that $x^{\varepsilon}$ is c\`adl\`ag)
\begin{equation}
\int_{0+}^{t}f\left(x_{s}^{\varepsilon}\right)\mathrm{d}x_{s}^{\varepsilon}=\int_{0}^{t}f\left(x_{s}^{\varepsilon}\right)\mathrm{d}x_{s}^{\varepsilon}=F\left(x_{t}^{\varepsilon}\right)-F\left(x_{0}^{\varepsilon}\right)-\sum_{0<s\le t}\left\{ \Delta F\left(x_{s}^{\varepsilon}\right)-f\left(x_{s}^{\varepsilon}\right)\Delta x_{s}^{\varepsilon}\right\} .\label{eq:trzecie}
\end{equation}
Using (\ref{eq:est_jumps}) and the assumption $\sum_{0<s\le t}\left(\Delta x_{s}\right)^{2}<+\infty$ we get by the dominated convergence that 
\begin{equation}
\sum_{0<s\le t}\left\{ \Delta F\left(x_{s}^{\varepsilon}\right)-f\left(x_{s}^{\varepsilon}\right)\Delta x_{s}^{\varepsilon}\right\} \ra \sum_{0<s\le t}\left\{ \Delta F\left(x_{s}\right)-f\left(x_{s}\right)\Delta x_{s}\right\} \label{eq:trzecie1}
\end{equation}
and 
\begin{equation}
\sum_{0<s\le t}\Delta f\left(x_{s}\right)\Delta x_{s}^{\varepsilon} \ra \sum_{0<s\le t}\Delta f\left(x_{s}\right)\Delta x_{s} \label{eq:trzecie2}
\end{equation}
as ${\varepsilon} \ra 0+$. Putting together relations (\ref{eq:jedne}), ({\ref{relation1}}), (\ref{eq:trzecie}), (\ref{eq:trzecie1}) and (\ref{eq:trzecie2}) we obtain the assertion.
\end{proof}

Theorem \ref{Foelmer} states a substitution formula for the  integral $\left({\cal X}\right)\int$
which does not coincide with the usual It\^o formula. Below, using the
family ${\cal X}$ we define another integral, denoted by $\left({\cal X}\right)'\int$,
which satisfies the usual It\^o formula. First, using integration by
parts, let us define for $t>0$ the integral 
\begin{align*}
\int_{0+}^{t}f\left(x_{s-}^{\varepsilon}\right)\dd{x_{s}}: & =f\left(x_{t}^{\varepsilon}\right)x_{t}-f\left(x_{0}^{\varepsilon}\right)x_{0}-\int_{0+}^{t}x_{s-}\dd{f\left(x_{s}^{\varepsilon}\right)}-\sum_{0<s\le t}\Delta x_{s}\Delta f\left(x_{s}^{\varepsilon}\right)
\end{align*}
where $\int_{0+}^{t}x_{s}\dd{f\left(x_{s}^{\varepsilon}\right)}$
is the usual Lebesgue-Stieltjes integral ($f\left(x^{\varepsilon}\right)$
has finite total variation on compacts). Now we define 
\[
\left({\cal X}\right)'\int_{0+}^{t}f\left(x_{s-}\right)\dd{x_{s}}:=\lim_{\varepsilon\ra0+}\int_{0+}^{t}f\left(x_{s-}^{\varepsilon}\right)\dd{x_{s}}
\]
if the limit on the right side exists. 
\begin{prop} \label{Foelmer1}
Assume that $x$ and $\cal{X}$ satisfy the assumptions of Theorem \ref{Foelmer}.
Then for any $f:\R\ra\R$ of class $C^{1}$ and $t>0$ the integral $\left({\cal X}\right)'\int_{0+}^{t}f\left(x_{s-}\right)\dd{x_{s}}$ exists and the usual It\^o
formula holds:
\begin{eqnarray*}
F\left(x_{t}\right)-F\left(x_{0}\right) & =& \left({\cal X}\right)'\int_{0+}^{t}f\left(x_{s-}\right)\dd{x_{s}}+\frac{1}{2}\int_{0+}^{t}f'\left(x_{s-}\right)\dd{\left\langle x\right\rangle _{s}} \\
 & &+\sum_{0<s\le t}\left\{ F\left(x_{s}\right)-F\left(x_{s-}\right)-f\left(x_{s-}\right)\Delta x_{s}\right\} ,
\end{eqnarray*}
where $F$ is an antiderivative of $f.$ 
\end{prop}
\begin{proof} We have 
\begin{align}
 & \int_{0+}^{t}x_{s-}\dd{f\left(x_{s}^{\varepsilon}\right)}=\int_{0+}^{t}x_{s}\dd{f\left(x_{s}^{\varepsilon}\right)}-\sum_{0<s\le t}\Delta x_{s}\Delta f\left(x_{s}^{\varepsilon}\right) \nonumber \\
 & =\int_{0+}^{t}\left(x_{s}-x_{s}^{\varepsilon}\right)\dd{f\left(x_{s}^{\varepsilon}\right)}+\int_{0+}^{t}x_{s}^{\varepsilon}\dd{f\left(x_{s}^{\varepsilon}\right)} \nonumber \\
 & \quad-\sum_{0<s\le t}\Delta x_{s}\Delta f\left(x_{s}^{\varepsilon}\right). \label{folm1}
\end{align}
In a similar way as in the proof of Theorem \ref{Foelmer} we prove that 
\begin{equation}
\int_{0+}^{t}\left(x_{s}-x_{s}^{\varepsilon}\right)\dd{f\left(x_{s}^{\varepsilon}\right)}\ra\frac{1}{2}\int_{0+}^{t}f'\left(x_{s-}\right)\dd{\left\langle x\right\rangle _{s}} \label{folm2}
\end{equation}
as $\varepsilon\ra0+$. By integration by parts and properties of
the Lebesgue-Stieltjes integral, 
\begin{align*}
\int_{0+}^{t}x_{s}^{\varepsilon}\dd{f\left(x_{s}^{\varepsilon}\right)} & =x_{t}^{\varepsilon}f\left(x_{t}^{\varepsilon}\right)-x_{0}^{\varepsilon}f\left(x_{0}^{\varepsilon}\right)-\left\{ F\left(x_{t}^{\varepsilon}\right)-F\left(x_{0}^{\varepsilon}\right)\right\} \\
 & \quad+\sum_{0<s\le t}\left\{ F\left(x_{s}^{\varepsilon}\right)-F\left(x_{s-}^{\varepsilon}\right)-f\left(x_{s-}^{\varepsilon}\right)\Delta x_{s}^{\varepsilon}\right\} .
\end{align*}
For $s\in\left(0,t\right]$, using (\ref{eq:est_jumps}), we estimate 
\begin{align*}
\left|F\left(x_{s}^{\varepsilon}\right)-F\left(x_{s-}^{\varepsilon}\right)-f\left(x_{s-}^{\varepsilon}\right)\Delta x_{s}^{\varepsilon}\right| & =\frac{1}{2}\left|f'\left(\tilde{x}_{s}^{\varepsilon}\right)\right|\left(\Delta x_{s}^{\varepsilon}\right)^{2},\\
 & \le\frac{1}{2}\left(\sup_{y\in\left[A_{t}-\varepsilon,B_{t}+\varepsilon\right]}\left|f'\left(y\right)\right|\right)K\left(\Delta x_{s}\right)^{2}
\end{align*}
where $\tilde{x}_{s}^{\varepsilon}\in\left[\min\left\{ x_{s-}^{\varepsilon},x_{s}^{\varepsilon}\right\} ,\max\left\{ x_{s-}^{\varepsilon},x_{s}^{\varepsilon}\right\} \right]$,
$A_{t}=\inf_{s\in\sbr{0,t}}x_{s}$, $B_{t}=\sup_{s\in\sbr{0,t}}x_{s}$.
Thus, by the dominated convergence,
\begin{equation} \label{folm3}
\sum_{0<s\le t}\left\{ F\left(x_{s}^{\varepsilon}\right)-F\left(x_{s-}^{\varepsilon}\right)-f\left(x_{s-}^{\varepsilon}\right)\Delta x_{s}^{\varepsilon}\right\} \ra\sum_{0<s\le t}\left\{ F\left(x_{s}\right)-F\left(x_{s-}\right)-f\left(x_{s-}\right)\Delta x_{s}\right\} .
\end{equation}
 Putting together relations (\ref{folm1})-(\ref{folm3}) we obtain the assertion. 
\end{proof} 
\begin{rem} The integral $\left({\cal X}\right)'\int$ satisfies the
``usual'' It\^o formula as in F\"{o}llmer's famous paper \cite{Foellmer:1981}. Unfortunately, it is not clear for the author of this paper if it is possible to ``match'' both approaches. This means if it is possible for a given sequence of partitions such that F\"{o}llmer's measures converge weakly to some measure corresponding to a quadratic variation,  to construct a family of functions ${\cal X}=\left(x^{\varepsilon}\right)_{{\varepsilon}>0}$ such that $\left\langle x\right\rangle$ defined in  (\ref{eq:q_var}) exists and is equal the continuous part of this quadratic variation. Opposite possibility for broad class of continuous paths $x$ and non-decreasing $\left\langle x\right\rangle$ follows easily from \cite[Theorem 7.1]{Obloj_local:2015}. Some recent (unpublished) results of the author and his collaborators (Jan Ob\l\'oj, David Pr\"omel and Pietro Siorpaes) indicate that the existence of the quadratic variation defined as the limit of normalized sequence of the truncated variations is weaker than the existence of the quadratic variation obtained as the limit of the sums of squares of the increments along so called Lebesgue partitions.
\end{rem}
As a direct consequence of Theorem \ref{Foelmer} and Proposition \ref{Foelmer1} we obtain the following corollary.
\begin{coro} \label{coroxxprim} Assume that $x$ and ${\cal X}$ satisfy the assumptions
of Theorem 2 and $f:\R\ra\R$ is of class $C^{1}$. Then for any $t>0$ both integrals
$\left({\cal X}\right)\int_{0+}^{t}f\left(x_{s-}\right)\dd{x_{s}}$
and $\left({\cal X}\right)'\int_{0+}^{t}f\left(x_{s-}\right)\dd{x_{s}}$
exist, moreover, we have the relation
\[
\left({\cal X}\right)'\int_{0+}^{t}f\left(x_{s-}\right)\dd{x_{s}}=\left({\cal X}\right)\int_{0+}^{t}f\left(x_{s-}\right)\dd{x_{s}}-\int_{0+}^{t}f'\left(x_{s-}\right)\dd{\left\langle x\right\rangle _{s}}.
\]
\end{coro}
\begin{rem} \label{Strat} Let $x:[0,+\infty)\ra\R$ be a c\`{a}dl\`{a}g path such that $\sum_{0<s\le t}\left(\Delta x_{s}\right)^{2}<+\infty$ for any $t > 0$. If the family ${\cal X}=\left(x^{\varepsilon}\right)_{{\varepsilon}>0}$ of c\`{a}dl\`{a}g paths $x^{\varepsilon}:[0,+\infty)\ra\R$
satisfies 
\begin{enumerate}[a)]
\item $\left\Vert x-x^{\varepsilon}\right\Vert _{\ns}\le\varepsilon$,
\item $x^{\varepsilon}$ has finite total variation on compacts,
\item there exists $K$ such that for any $t>0$, $\left|\Delta x_{t}^{\varepsilon}\right|\le K\left|\Delta x_{t}\right|$,
\end{enumerate}
 then it follows from (\ref{eq:trzecie}) and (\ref{folm3}) that for any $f:\R\ra\R$ of class $C^{1}$ and $t>0$  the integral
\[
\left({\cal X}\right)''\int_{0+}^{t}f\left(x_{s-}\right)\dd{x_{s}}:=\lim_{\varepsilon\ra0+}\int_{0+}^{t}f\left(x_{s-}^{\varepsilon}\right)\dd{x_{s}^{\varepsilon}}.
\]
is well defined. This integral corresponds to the Stratonovich stochastic integral and, whenever assumptions of Theorem 2 are satisfied, then we have 
\begin{align*}
\left({\cal X}\right)''\int_{0+}^{t}f\left(x_{s-}\right)\dd{x_{s}} & = \left({\cal X}\right)'\int_{0+}^{t}f\left(x_{s-}\right)\dd{x_{s}} + \frac{1}{2}\int_{0+}^{t}f'\left(x_{s-}\right)\dd{\left\langle x\right\rangle _{s}} \\ 
& = \left({\cal X}\right)\int_{0+}^{t}f\left(x_{s-}\right)\dd{x_{s}} - \frac{1}{2}\int_{0+}^{t}f'\left(x_{s-}\right)\dd{\left\langle x\right\rangle _{s}}.
\end{align*}
\end{rem}
Recall that $\mathbb{D}$ denotes the family of  c\`{a}dl\`{a}g functions $x:\left[0,+\ns\right)\ra\R.$
Using Proposition \ref{Foelmer1}
 we easily obtain the following result linking integrals $\rbr{\cal{X}} \int$, $\rbr{\cal{X}}' \int$, $\rbr{\cal{X}}'' \int$ and the It\^o stochastic integral. 
\begin{coro} Let $X_{t}$, $t\ge0$, be a  c\`{a}dl\`{a}g semimartingale on a
filtered probability space $\rbr{\Omega,\mathbb{F},\P}$ and let $\left\langle X\right\rangle $
denote the continuous part of its quadratic variation. Assume that $\left(S^{\varepsilon}\right)_{\varepsilon>0}$ is a sequence of mappings $S^{\varepsilon}:\mathbb{D\ra\mathbb{D}}$ such that for any $x\in\mathbb{D}$ and $\varepsilon>0$, $x^{\varepsilon}:=S^{\varepsilon}\left(x\right)$
satisfies conditions (a)-(c) of Remark \ref{Strat}  and is such that for almost
all $\omega\in\Omega$ the sequence ${\cal X}=\left(S^{\varepsilon}\left(x\right)\right)_{\varepsilon>0}$,
where $x=X\left(\omega\right)$, satisfies 
\begin{equation}
2\lim_{\varepsilon\ra0+}\int_{0}^{t}\left(x_{s}-x_{s}^{\varepsilon}\right)\dd{x_{s}^{\varepsilon}}=\left\langle X\right\rangle_t \left(\omega\right)\text{ for any }t>0.\label{eq:q_var_assumption}
\end{equation}
If $f:\R\ra\R$ is of class $C^{1}$ and $t>0$ then for almost all
$\omega\in\Omega$ the integrals $\left({\cal X}\right)\int_{0+}^{t}f\left(x_{s-}\right)\dd{x_{s}}$, $\left({\cal X}\right)'\int_{0+}^{t}f\left(x_{s-}\right)\dd{x_{s}}$
and $\left({\cal X}\right)''\int_{0+}^{t}f\left(x_{s-}\right)\dd{x_{s}}$
exist and satisfy
\[
\left({\cal X}\right)\int_{0+}^{t}f\left(x_{s-}\right)\dd{x_{s}}=\left(\int_{0+}^{t}f\left(X_{s-}\right)\dd{X_{s}}\right)\left(\omega\right)+\rbr{\int_{0+}^{t}f'\left(X_{s-}\right) \dd \left\langle X\right\rangle_s }\left(\omega\right),
\]
\begin{equation}
\left({\cal X}\right)'\int_{0+}^{t}f\left(x_{s-}\right)\dd{x_{s}}=\left(\int_{0+}^{t}f\left(X_{s-}\right)\dd{X_{s}}\right)\left(\omega\right)\label{eq:Ito1}
\end{equation}
and
\[
\left({\cal X}\right)''\int_{0+}^{t}f\left(x_{s-}\right)\dd{x_{s}}=\left(\int_{0+}^{t}f\left(X_{s-}\right)\dd{X_{s}}\right)\left(\omega\right)+\frac{1}{2}\rbr{\int_{0+}^{t}f'\left(X_{s-}\right) \dd \left\langle X\right\rangle_s }\left(\omega\right),
\]
where $\int_{0+}^{t}f\left(X_{s-}\right)\dd{X_{s}}$ denotes the usual
It\^o stochastic integral.
\end{coro}
\begin{proof} Let $F$ be an antiderivative of $f$. Using the It\^o formula
we have a.s. 
\begin{align}
\int_{0+}^{t}f\left(X_{s-}\right)\dd{X_{s}} & =F\left(X_{t}\right)-F\left(X_{0}\right)-\frac{1}{2}\int_{0+}^{t}f'\left(X_{s-}\right)\dd{\left\langle X\right\rangle _{s}}\nonumber \\
 & \quad-\sum_{0<s\le t}\left\{ F\left(X_{s}\right)-F\left(X_{s-}\right)-f\left(X_{s-}\right)\Delta X_{s}\right\} .\label{eq:Ito}
\end{align}
Let $\omega\in\Omega$ be such that for $x=X\left(\omega\right)$
the condition (\ref{eq:q_var_assumption}) is true. By (\ref{eq:q_var_assumption})
and by Proposition \ref{Foelmer1} we have 
\begin{align*}
\left({\cal X}\right)'\int_{0+}^{t}f\left(x_{s-}\right)\dd{x_{s}} & =F\left(x_{t}\right)-F\left(x_{0}\right)-\frac{1}{2}\int_{0+}^{t}f'\left(x_{s-}\right)\dd{\left\langle x\right\rangle _{s}}\\
 & \quad-\sum_{0<s\le t}\left\{ F\left(x_{s}\right)-F\left(x_{s-}\right)-f\left(x_{s-}\right)\Delta x_{s}\right\} 
\end{align*}
where $\left\langle x\right\rangle =\left\langle X\right\rangle \left(\omega\right)$.
Taking the intersection of the subsets of $\Omega$ where (\ref{eq:Ito})
and (\ref{eq:q_var_assumption}) resp. hold, we get a subset of $\Omega$
of probability $1$ where equality (\ref{eq:Ito1}) holds. The relations for $\rbr{\cal{X}}\int$ and $\rbr{\cal{X}}''\int$ follow now easily from (\ref{eq:q_var_assumption}), Corollary \ref{coroxxprim} and Remark \ref{Strat}.
\end{proof}

\section*{Acknowledgements}
The author would like to thank the anonymous referees for their valuable comments which led to
improvement of this work.

\section*{Competing interests}
The author declares to have no competing interests.

\section*{Funding}
This research was partially founded by the National Science Centre, Poland, under Grant No.~$2016/21/$B/ST$1/01489.$

\bibliographystyle{tfs}
\bibliography{/Users/rafallochowski/biblio/biblio}

\end{document}